\documentclass[a4paper,twoside,12pt, fleqn]{article}
\usepackage[fleqn]{amsmath}
\usepackage{amsfonts}
\usepackage{latexsym,amsthm, amssymb}
\usepackage[latin1]{inputenc}
\usepackage{enumerate}
\usepackage{longtable}
\numberwithin{equation}{section}
\newtheorem{lemma}{Lemma}
\newtheorem{thm}{Theorem}

\date{}

\author {\small T.~Royen\\[-0.8ex]
\small Fachhochschule Bingen, University of Applied Sciences,\\[-0.8ex]
\small Berlinstra\ss e 109, D-55411 Bingen, Germany\\[-0.8ex]
\small  E-mail: \textsf{~thomas.royen@t-online.de}}

\title{\normalsize THE DISTRIBUTION OF THE SQUARE SUM OF DIRICHLET RANDOM VARIABLES AND A TABLE WITH QUANTILES OF GREENWOOD'S STATISTIC}

\begin{document}
 \maketitle

 \begin{abstract}
The exact distribution of the square sum of Dirichlet random variables is given by two different univariate integral representations. Alternatively, three representations by orthogonal series with Jacobi or Legendre polynomials are derived. As a special case the distribution of the square sum $U_n^2$ of spacings - also called Greenwood's statistic - is obtained. Nine quantiles of $nU_n^2-1$   are tabulated with eight digits for $n$ from 10 to 100.\\[0.5cm]
\end{abstract}

 \section{Introduction and Notation}
Greenwood's statistic is defined by

\begin{displaymath}
U^2_{k+1}= \sum^{k+1}_{i=1}\left( U_{i:k}- U_{i-1 : k}\right)^2
\end{displaymath}
where the $U_{i : k}$  denote the successive ordered values of a random sample $U_1, \ldots, U_k$  from a uniform distribution on $(0,1)$  with $U_{0 : k}= 0$ and $ U_{k+1 : k}=1$. If $F_0$ denotes any completely specified continuous cdf then, with a random sample $X_1, \ldots X_k$ and   $U_i=F_0 (X_i)$, the statistic $U_{k+1}^ 2$  can be used for a test of the hypothesis that the unknown cdf $F$ of $X$ coincides with $F_0$. For the efficiency of this test see e.g. Pyke (1965).
For the application of Greenwood's statistic by a test for exponentiality see e.g. Subhash and Gupta (1988).\\

\par The computation of the distribution of $U_n^2$  has a long history beginning with Moran (1947)
 and (1951),
 Gardner (1952)
  followed by Hill (1979),
 Burrows (1979),
  Currie (1981),
  Stephens (1981)
 Currie and Stephens (1986),
  Does et al. (1988),
  Ghosh and Jammalamadaka (1998)
  and (2000).
  The exact methods of Burrows, Currie and Stephens use iterated integrations up to $n = 20$. The subsequent papers are concerned with approximations mainly by Edgeworth expansions and saddle-point methods. It is well known that the distribution of the standardized $U_n^2$ converges very slowly to a normal distribution. For the asymptotic normality of a more general class of statistics, based on spacings, see Darling (1953).
   A comparison of the approximate quantiles with the available ``exact'' quantiles for small sample sizes e.g. in Ghosh and Jammalamadaka (2000)
  shows rather unsatisfying discrepancies. However, thanks to the increasing computing power of computer algebra systems, computations are available based on exact representations of the cdf of Greenwood's statistic. Thus errors are reduced solely to those of the numerical evaluation.\\

\par In this paper more generally some exact representations are given for the distribution of the square sum $U_n^2$ of Dirichlet random variables derived from gamma distributed variables of order $\alpha$. This statistic can be used as a competitor with the maximum likelihood statistic for the test of a specified value $\alpha_0$  of the unknown order parameter $\alpha$  of a gamma distribution with any unknown scale parameter. For the maximum likelihood estimation of $\alpha$  see  Provost (1988)
 and chapter 17 in Kotz et al. (1994).
  Greenwood's distribution is obtained by the special case  $\alpha= 1$. In section 4 the cdf of $U_n$  is represented by orthogonal series with Jacobi or Legendre polynomials. In section 5 two different univariate integral representations for the cdf of $U_n^2$  are found, derived by the Fourier and Laplace inversion formula. Theorem 2, the main result, provides also the more general case with gamma random variables of different orders $\alpha_i$  by a univariate integral over a product of parabolic cylinder functions. Theoretically, the integral representations look more elegant since no special coefficients have to be computed before. For the orthogonal series a sequence of moments must be computed, but the computing effort can be reduced to two linear recursion formulas. The eight digit quantiles of Greenwood's statistic for $n$ from 10 to 100 in section 6 were computed by the orthogonal series in (\ref{eq_4.6}) and again by (\ref{eq_4.7}) . In spite of the very high computing accuracy, required to achieve the eight digit quantile values, this method works more quickly than the integral representations in section 5.

\par Now let $X_1, \ldots, X_n$  be i.i.d. random variables with a gamma pdf
\begin{equation}    \label{eq_1.1}
     g_a (x) = \left( \Gamma (\alpha)\right)^{-1} x^{a - 1} e^{- x} \  \textrm{and} \  G_a (x) = \int_0^x g_a (\xi) d \xi \, , \,  \alpha > 0 \, , \ x > 0 \, .
\end{equation}
The ``sample triangle'' in $\mathbb{R}^n$  generated by $X_1, \ldots, X_n$ , has the vertices $(0, \ldots , 0), (X_1, \ldots , X_n), \ ( \bar{X}, \ldots , \bar{X})$,
with the sample mean $\bar{X}$ and the angle $\Phi$  between the vectors ($X_1, \ldots, X_n$ ) and  $(1, \ldots, 1)$. We use the further
notations
\begin{eqnarray}
       Y &=& \sum_{i=1}^n X_i = n \bar{X} \, , \quad  \ Z=R^2= \sum_{i=1}^n X_i^2 \, , \nonumber \\
       U &=& \frac{R}{Y}= \frac{1}{\sqrt{n} \cos \Phi}\, , \ \ Q=R^2 - n \bar{X}^2 = n S^2 \, . \label{eq_1.2}
\end{eqnarray}
The random variables $Y_i = X_i /Y , \  i = 1\, , \ldots \, , n-1$  have the joint Dirichlet density
\begin{displaymath}
       \frac{\Gamma (\alpha n)}{(\Gamma (\alpha))^n} \left( 1 - \textstyle{\sum_{i=1}^{n-1}} y_i \right)^{\alpha - 1} \textstyle{\prod_{i=1}^{n-1}}y_i^{\alpha-1} \, .
\end{displaymath}
Thus, the distribution of the square sum \mbox{$U_n^2 = \sum_{i=1}^n Y_i^2  \, ,  \ Y_1 + \ldots + Y_n =1$}  is closely related to the distribution of the squared sample coefficient of variation
\begin{equation}
         S^2 /\bar{X}^2= \tan^2 (\Phi) = n U_n^2 -1 \, . \label{eq_1.3}
\end{equation}
By Laha (1954)
 and Lukacs (1955)
  the independence of $\bar{X}$  and $S/\bar{X}$ was shown to be a characterization of the gamma distribution. We shall use the independence of $U_n$ and $Y=n \bar{X}$.

 \par Apart from some special notations as e.g. in (\ref{eq_1.1}) the cdf, pdf, char. function (cf) and the Laplace transform (Lt) of any continuous random variable $Z$ will be denoted by $F_Z\, , \ f_Z \, ,  \ \hat{f}_Z$  and $f_Z^\ast$  respectively.   $E_Z (\ldots )$ means expectation with regard to $Z$. The difference operator is denoted by $\Delta$, i.e. $\Delta a_k = a_{k+1}- a_k$  for any sequence $(a_k)$ and $(c)_k = \prod_{i=0}^{k-1}(c+i)$.  Formulas from Abramowitz and Stegun are denoted by A.S. and their number.

 \section{The distribution of $R$ and $Q$}
 These distributions were already given in Royen (2007a),
  (2007b).
  The coefficients, needed for a series representation of the cdf $F_R$ of $R$, are also used in the subsequent sections. Therefore, this representation is derived here more concisely.
\par For a given power series $f(z) = \textstyle{\sum_{k=0}^\infty} a_k z^k , a_0 >0$,  the coefficients $a_{n,k}$ of $\left(f (z)\right)^n$  can be computed by two steps: Setting
\begin{displaymath}
       \log \left(f (z)\right) = \log\left(a_0\right)+ \textstyle{\sum_{k=1}^\infty} b_k z^k / k\, ,
\end{displaymath}
the relation
\begin{eqnarray}{}
          f' (z)&=& \textstyle{\sum_{k= 0}^\infty} (k+1)a_{k+1}z^k =
          f(z) \textstyle{\sum_{k=0}^\infty} b_{k+1} z^k \nonumber\\
          &=& \textstyle{\sum_{k=0}^\infty} \left( \textstyle{\sum_{j= 0}^k} a_j b_{k+1-j}\right)z^k \ \mbox{ implies}\nonumber\\
      b_{k+1}&=& a_0^{-1}\left((k+1)a_{k+1}- \textstyle{\sum_{j=1}^k} a_j b_{k+1-j}\right)\, , k \in \mathbb{N}_0 \, , \ \mbox{and}\nonumber\\
a_{n,0}&=& a_0^n \, , \ \ a_{n, k+1}= \textstyle{\frac{n}{k+1} \sum_{j=0}^k} a_{n,j}b_{k+1-j}\, ,  k \in \mathbb{N}_0 \ .      \label{eq_2.1}
\end{eqnarray}

 Let be $C=\cos (\Phi)$. The moments
\begin{equation}
       \gamma_{\alpha n +k}:= E \left( C^{\alpha n+k} \right)     \label{eq_2.2}
\end{equation}
depend on $\alpha$ and $n$ not only by $\alpha n$, but we write $\gamma_{\alpha n+k}$ for simplicity instead more precisely $\gamma_{\alpha ,n, \alpha n+k}$.
With $U^{-1} = \sqrt{n}C$ it follows from the independence of $U$ and $\bar{X}$ that
\begin{eqnarray}
F_R (r) & = &  E_U \left( G_{\alpha n} (r U^{-1}) \right)\nonumber \\
& = & E_C \left( G_{\alpha n} (r C\sqrt{n}) \right) = \frac{(r \sqrt{n})^{\alpha n}}{\Gamma (\alpha n)} \sum_{k= 0}^\infty \frac{\gamma_{\alpha n+k}}{\alpha n +k} \frac{(-r \sqrt{n})^k}{k!}\, .     \label{eq_2.3}
\end{eqnarray}
On the other hand we have with
 \begin{eqnarray}
        a_{\alpha, k} &=& \frac{
        \Gamma \left( (\alpha+k)/2\right)}{2 \Gamma (\alpha) k!}\ \mbox{the Lt} \nonumber \\[.8ex]
        f_Z^\ast (t) &=& \left( f_{X^2}^\ast (t)\right)^n = \left( t^{-\alpha/2} \textstyle{\sum_{k=0}^\infty} (-1)^k a_{\alpha , k} t^{- k/2}\right)^n \nonumber \\[1.5ex]
        &=& t^{-\alpha n/2} \textstyle{\sum_{k=0}^\infty} (-1)^k a_{\alpha , n,k}  t^{- k/2} \ ,     \label{eq_2.4}
 \end{eqnarray}
where the $a_{\alpha , n,k}$ are computed from the $a_{\alpha ,k}$ by the above procedure leading to (\ref{eq_2.1}). With $Z = R^2$ we obtain from (\ref{eq_2.4}) by inversion, followed by integration
\begin{equation}
       F_R (r) = r^{\alpha n} \sum_{k=0}^\infty \frac{a_{\alpha,n,k}}{\Gamma \left( 1+ (\alpha n + k)/2\right)}(-r)^k \, ,     \label{eq_2.5}
\end{equation}
and by comparison with (\ref{eq_2.3})
\begin{equation}
       \gamma_{\alpha n + k}= \frac{2 \Gamma (\alpha n)k!}{\Gamma \left( (\alpha n+k)/2\right)n^{(\alpha n +k)/2 }} a_{\alpha,n,k}\ .       \label{eq_2.6}
 \end{equation}
Finally, from the density
\begin{eqnarray*}
       f_R (r) &=& E_\Phi \left( \sqrt{n}\cos \Phi \cdot g_{\alpha n} ( r \sqrt{n} \cos \Phi ) \right)\\
       &=& \sqrt{n}g_{\alpha n} (r \sqrt{n}) E_\Phi \left( (\cos \Phi)^{\alpha n} e^{r \sqrt{n} (1-\cos \Phi )} \right)\\
       &=& \sqrt{n} g_{an} (r\sqrt{n}) \sum_{k=0}^\infty E_\Phi \left( (\cos \Phi)^{\alpha n}(1-\cos \Phi)^k  \right)\frac{(r \sqrt{n})^k}{k!}\\
       &=& \sqrt{n} \sum_{k=0}^\infty \textstyle{\alpha n + k -1 \choose k} (-\Delta)^k \gamma_{\alpha n} \cdot g_{\alpha n + k}(r \sqrt{n})
\end{eqnarray*}
we obtain  $F_R$ as a convex combination of gamma distribution functions:
\begin{equation}
        F_R (r) = \sum_{k=0}^\infty p_{\alpha,n,k} G_{\alpha n +k} (r \sqrt{n}) \, , \ \ p_{\alpha,n,k}  = \textstyle{\alpha n + k -1 \choose k} (-\Delta)^k \gamma_{\alpha n} \, .      \label{eq_2.7}
\end{equation}
With gamma$(\alpha n+k)$-distributed random variables $Y_{\alpha n+k}$ and a random index $K$ with $P\{ K=k\} = p_{\alpha,n,k}$ this can be restated more concisely as a stochastic representation:
\begin{equation}
       R \ \textrm{is distributed as } n^{-1/2}   Y_{\alpha n + K} \ .     \label{eq_2.8}
\end{equation}
   The distribution of $Q=nS^2$ could be derived from the distribution of  $U_n^2$ , but we have directly with
\begin{displaymath}
       K_\beta (x,y) = \sum_{k=0}^\infty \frac{He_k (y)}{\Gamma \left( \beta +  k/2 \right)} \frac{(-x)^k}{k!} \, ,
\end{displaymath}
where $He_k$ denotes the Hermite polynomials (A.S.22.2.15), and the parabolic cylinder functions $D_{-\alpha}$  the integral representation
 \begin{equation}
       \!\!\!\!\!\!\!\!\!\!\!\! F_Q(x) = \frac{1}{\sqrt{2 \pi}}\left( \frac{x}{2} \right)^{\alpha n/2}\!\!\!  \int_{- \infty}^{\infty}   K_{\alpha n/2+1} \textstyle( {\sqrt{nx/2}, y} ) \! \left(  D_{-\alpha} (- y/\sqrt{n})\right)^n e^{-y^2/4} dy \, ,                                  \label{eq_2.9}
\end{equation}
which was given in Royen (2007b)
 with slightly different notations. Further representations of $F_Q$ are found in Royen (2007a).
\section{Some Laplace transforms}
If $X$ has a gamma density $g_\alpha$, then the bivariate Lt of $(X, X^2)$ is given by
\begin{eqnarray}
       E \left( e^{-sX-tX^2}\right)  &=& \frac{1}{t^{\alpha/2}\Gamma (\alpha)} \int_0^\infty x^{\alpha-1} e^{-x^2} e^{-x(1+s)/\sqrt{t}}dx  \nonumber \\
       &=& \frac{1} {2t^{\alpha/2}\Gamma (\alpha)} \sum_{k=0}^\infty
       \frac{\Gamma
       \left( (\alpha + k)/2\right)}{k!}
       \left( - \frac{1+s}{\sqrt{t}}  \right)^k  \nonumber     \\
       &=& \frac{1}{(2t)^{\alpha/2}}\exp
       \left(  \frac{(1+s)^2}{8t}  \right)D_{-\alpha} \left( \frac{1+s}{\sqrt{2t}}  \right)   \label{eq_3.1}
\end{eqnarray}
with the parabolic cylinder function $D_{-\alpha}$ from (A.S.19.3.1). If $\alpha \in \mathbb{N}$  then
 \begin{equation}
        \!\!\!\!\!\! 2^{-\alpha/2}\exp
        \left( \frac{z^2}{8}    \right)
        \!\! D_{-\alpha} \!\! \left( \frac{z}{\sqrt{2}}   \right) =
        {\frac{\sqrt{\pi}}{2 \Gamma (\alpha)}}
        \left( - \frac{d}{dz}\right)^{\alpha-1}
        \!\! \left( \exp   \left( \frac{z^2}{4} \right)  er\!f\!c
        \left(  \frac{z}{2} \right) \right)
        \label{eq_3.2}
 \end{equation}
and in particular
\begin{displaymath}
\exp \left(  \frac{z^2}{4}\right) D_{-1}(z) = \sqrt{\frac{\pi}{2}} \exp \left(  \frac{z^2}{2}\right) er\!f\!c  \left(  \frac{z}{\sqrt{2}} \right) \, .
\end{displaymath}
>From (\ref{eq_1.2}), (\ref{eq_2.2}), (\ref{eq_2.7}) we obtain
\begin{eqnarray*}
       F_R(r) &=& P\{ UY \leq r \}= \int_0^{r\sqrt{n}} F_U (ry^{-1}) g_{\alpha n} (y) dy \\
       &=& \frac{1}{\Gamma (\alpha n)}\int_0^{r\sqrt{n}} \left( 1- F_C ( n^{-1/2} r^{-1} y ) \right) y^{\alpha n -1} e^{-y}dy  \\
       &=&\frac{(r\sqrt{n})^{\alpha n}}{\Gamma (\alpha n)}\int_0^1 (1 - F_C (x)) x^{\alpha n-1}e^{-xr\sqrt{n}} dx \, .
\end{eqnarray*}
Thus,
\begin{equation}
       \Gamma (\alpha n) s^{-\alpha n}F_R \left( n^{-1/2} s \right) =
       \Gamma (\alpha n) s^{-\alpha n} \sum_{k=0}^\infty p_{\alpha , n , k} G_{\alpha n + k}(s)
       \label{eq_3.3}
\end{equation}
is the Lt of     $x^{\alpha n-1} (1 - F_C (x)), 0 < x \leq1$. $(F_C(x) = 0 \ \ if  \  x \leq n^{-1/2})$.\\

\par Also the Lt of $C= \cos (\Phi)$ can be represented by means of the probabilities $p_{\alpha , n , k}$  from (\ref{eq_2.7}). If $Z_1$ denotes a  r.v. with density  $e^{-z}$, independent of $U$, and $X_{\alpha n -1}$  a  r.v. with a beta$(1, \alpha n -1)$-distribution, independent of $R$, then
 \begin{displaymath}
        \frac{E (X_{\alpha n -1}^k)}{E(Z_1^k)}= \frac{(\alpha n-1) B (k + 1 , \alpha n -1 )}{k!}=
        \frac{\Gamma (\alpha n)}{\Gamma (\alpha n + k)}= \frac{1}{(\alpha n)_k}\, .
\end{displaymath}
Comparing this with $E(R^k) = E(UY)^k = E(U^k) E(Y^k) = E(U^k) (\alpha n)_k$  it follows that $UZ_1$  and $RX_{\alpha n-1}$  have the same moments which determine the distribution completely. Thus
 \begin{equation}
 UZ_1 \textrm{ is distributed as }RX_{\alpha n -1}
 \label{eq_3.4}
 \end{equation}
and consequently the Lt $f_C^\ast$  of $C = \cos (\Phi)$  is given by
\begin{eqnarray}
       f_C^\ast(s) &=& E \left(e^{-sC} \right)
       = E \left(e^{-s/ (\sqrt{n} \cdot U)} \right) =  P \{ UZ_1 > sn^{-1/2}\} \nonumber \\
       &=& P \{ RX_{\alpha n-1} > sn^{-1/2}\}
       = P \{ X_{\alpha n-1}Y_{\alpha n+K}> s\} \nonumber \\[1mm]
       &=& \textstyle{\sum_{k=0}^\infty} p_{\alpha, n, k} \int_s ^\infty (1-sy^{-1})^{\alpha n-1}g_{\alpha n+k}(y)dy \nonumber \\[1mm]
       &=&\textstyle{ \sum_{k=0}^\infty} \left( (- \Delta)^k  \gamma_{\alpha n} \right) \frac{1}{k! \Gamma (\alpha n)} \int_s^\infty (y-s)^{\alpha n-1}y^k e^{-y}dy \nonumber \\[1mm]
       &=& e^{-s }\textstyle{\sum_{k=0}^\infty} \left( (- \Delta)^k \gamma_{\alpha n} \right) \frac{1}{k!}\int_0^\infty (y+s)^k g_{\alpha n} (y) dy \nonumber \\[1mm]
       &=& e^{-s} \textstyle{\sum_{k=0}^\infty} \left( (- \Delta)^k \gamma_{\alpha n} \right) \textstyle{\sum_{j=0}^k} {\alpha n +k-j-1 \choose k-j} s{^j}/{j!} \nonumber \\[1mm]
          &=& e^{-s} \textstyle{\left( 1 + \sum_{k=1}^\infty p_{\alpha, n, k} {\alpha n + k  - 1 \choose k}^{-1}
          \sum_{j=1}^k {\alpha n + k  -j - 1 \choose k - j} s^j/j! \right)} .  \label{eq_3.5}
\end{eqnarray}

>From the binomial series we obtain the mgf of  $\log (\sqrt{n}U)$ :
\begin{eqnarray}
       E (\sqrt{n}U)^s &=& E (\cos (\Phi))^{-s} = E \left( \frac{(\cos (\Phi))^{\alpha n}}{(1-(1-\cos (\Phi)))^{\alpha n+s}}  \right) \nonumber \\[1ex]
       &=& \textstyle{\sum_{k=0}^\infty}  ((-\Delta)^k \gamma_{\alpha n}) \frac{(\alpha n + s)_k}{k!}  = \textstyle{\sum_{k=0}^\infty}
       p_{\alpha, n, k} \frac{(\alpha n + s)_k}{(\alpha n)_k} \, , \label{eq_3.6}
       \end{eqnarray}
following also from (\ref{eq_2.8}) and $\hat{f}_{\log Y_{\alpha n + K}} = \hat{f}_{\log Y} \hat{f}_{\log (\sqrt{n}U)}$.\\

In particular for  $\alpha=1$ the ``missing moments''  $\gamma_m$, $m=1,\ldots n-1$,  are given by
\begin{equation}
       \gamma_m = \gamma_n + \textstyle{\sum_{k=1}^\infty} p_{1, n, k} \frac{(n-m)_k}{(n)_k} \ \textrm{ or }\
        \gamma_{m-1}  =   \textstyle{\sum_{k=0}^\infty} (-\Delta)^k \gamma_m \, ,  \  m=n, n-1, \ldots 2\, .                                  \label{eq_3.7}
\end{equation}
The inversion formula, derived from the Bernstein polynomials, (see e.g. Feller (1971)
 would already provide
\begin{equation}
       F_C(x) = \lim_{n \to \infty}  \textstyle{\sum_{k \leq nx}} {n \choose k} (-\Delta)^{n-k} \gamma_k \, ,\label{eq_3.8}
\end{equation}
but for a satisfying accuracy a very large $n$ would be required. Therefore, some series expansions with orthogonal polynomials will be given in the subsequent section.

\section{Some representations of the distribution of $U$ by orthogonal polynomials}

The square integrability of the densities $f_C$  and $f_U$  is equivalent and gua\-rantees the uniform convergence of the following orthogonal series. The condition (\ref{eq_4.4}) of the subsequent lemma is sufficient for the square integrability of $f_U$.
\ \\[-3ex]
\begin{lemma}  Let  $Y_1, \ldots Y_n$ be random variables with $\sum_{i=1}^n Y_i=1$  and a joint Dirichlet distribution with the parameters $\alpha_1, \ldots \alpha_n$.  Then the cdf $F_U$  of $U=  \left( \sum_{i=1}^n Y_i^2 \right)^{1/2}$ satisfies the following relations:
\end{lemma}
\ \\[-5ex]
\begin{gather}
F_U \left( n^{1/2} + \varepsilon \right)
\simeq \frac{\Gamma (\alpha)}{\prod_{i=1}^n \Gamma (\alpha_i)}
\frac{(2 \pi)^{(n-1)/2}} {\Gamma\left(( n+1)/2\right)} n^{(3n-1)/4-\alpha} \cdot \varepsilon^{(n-1)/2} , \label{eq_4.1} \\[1.5ex]
\qquad \qquad \quad \ \alpha = \textstyle{\sum_{i=1}^n} \alpha_i\, , \quad
\varepsilon \to 0 \, . \nonumber
\end{gather}
\ \\[-5.5ex]
\begin{gather}
\sum_{i=1}^n
P \left\{ Y_i \geq 1- \varepsilon - \frac{\varepsilon^2}{2 (n-1)}  \right\} < 1 - F_U (1 - \varepsilon) \leq  \nonumber\\
\qquad \leq
\sum_{i=1}^n
P\left\{ Y_i \geq \frac{1}{2} + \frac{1}{2}
(1 - 4 \varepsilon + 2 \varepsilon^2)^{1/2}\right\} \, , \label{eq_4.2}\\
0 < \varepsilon < 1-2^{-1/2} , \textrm{where the } Y_i \textrm{ are beta} (\alpha_i , \alpha - \alpha_i) \textrm{-distributed.} \notag\\
1 - F_U (1 - \varepsilon) \simeq
\frac{
 \# \left\{ i | \alpha_i = \alpha_{\mathrm{max}} \right\}}  {B (\alpha_{\mathrm{max}} , \alpha - \alpha_{\mathrm{max}})}
\cdot \frac
{\varepsilon^{\alpha - \alpha_{\mathrm{max} } }}
{\alpha - \alpha_{\mathrm{max} }} \, , \ \ \varepsilon \to 0 \, .\label{eq_4.3}\\[0.5ex]
\!\!\!\! \textrm{The density } f_U \textrm{ is square integrable if } n > 2 \textrm{  and } \alpha - \alpha_{\mathrm{max}}    > 1/2 \, . \label{eq_4.4}
\end{gather}
Remark: For identical $\alpha_i=\alpha$  the $\alpha$  in the lemma has to be replaced by $\alpha n$ and we obtain the conditions
\begin{gather}
       f_U \textrm{ is square integrable if } n > \mathrm{max} \left( 2, 1 + \frac{1}{2 \alpha} \right)  ,\nonumber \\
         f_U \!\left( n^{-1/2}\right) = f_U (1)= 0 \textrm{ if } n > \mathrm{max}  \left( 3, 1 + \frac{1}{\alpha} \right) \, . \label{eq_4.5}
\end{gather}
\begin{proof}  With $x_i = (1 + \varepsilon_i) \bar{x} = (1 + \varepsilon_i) yn^{-1}$  we have $\tan^2(\varphi) = \frac{1}{n} \sum_{i=1}^n \varepsilon_i^2$  and consequently with $t \to 0$ :
\begin{gather*}
       P \{ \tan (\Phi) \leq t \} = \left( \textstyle{\prod_{i=1}^n} \Gamma (\alpha_i) \right)^{-1} \int_{\tan (\Phi) \leq t} \textstyle{\prod_{i=1}^n} e^{-x_i}x_i^{\alpha_i-1}dx_i \simeq \\
       \simeq b_{n-1} \left( \textstyle{\prod_{i=1}^n} \Gamma (\alpha_i) \right)^{-1} \int_0^\infty e^{-y} (yn^{-1})^{\alpha-n} (yn^{-1/2})^{n-1} n^{-1/2}
       dy \cdot t^{n-1}=\\
       = b_{n-1} \left( \textstyle{\prod_{i=1}^n}  \Gamma (\alpha_i) \right)^{-1} \Gamma (\alpha) n^{n/2-\alpha} t^{n-1} \, ,
\end{gather*}
where $b_{n-1}$  denotes the volume $\pi^{(n-1)/2} \bigl( \Gamma  \left((n+1)/2 \right)\bigr)^{-1}$ of the $(n\!-\!1)$-unit ball. Thus, (\ref{eq_4.1}) follows from
\begin{displaymath}
P \left\{ U \leq n^{-1/2}+ \varepsilon \right\} = P \left\{ \tan^2 (\Phi) \leq 2 \varepsilon \sqrt{n} + n\varepsilon^2 \right\} \simeq P \left\{ \tan(\Phi) \leq n^{1/4} \sqrt{2 \varepsilon} \right\} \, .
\end{displaymath}
Now let be $Y_n = \mathrm{max } \, Y_i$ and $0 < \delta < {1/2}$.  Then
\begin{gather*}
Y_n = 1- \delta \Rightarrow \textstyle{\sum_{i=1}^{n-1}} Y_i = \delta \Rightarrow
\frac{\delta^2}{n-1} \leq \textstyle{\sum_{i=1}^{n-1}} Y_i^2 \leq \delta^2 \Rightarrow \\
\left(  (1-\delta)^2 + \frac{\delta^2}{n-1} \right)^{1/2} \leq U \leq \left(  (1-\delta)^2 + \delta^2 \right)^{1/2} \, .
\end{gather*}
Conversely,  from $Y_n = \mathrm{max } \, Y_i$  and $U=1-\varepsilon \, , \ \varepsilon < 1-2^{-1/2}$,   it follows that $Y_n = 1- \delta$  with
\begin{gather*}
\varepsilon + \frac{1}{2 (n-1)} \varepsilon^2 <
\delta_1 (\varepsilon) := \frac{n-1}{n}-\frac{n-1}{n} \left( 1- \frac{2n}{n-1} \varepsilon (1 - \varepsilon/2)   \right)^{1/2} \leq \\
\leq \delta \leq \delta_2 (\varepsilon) := \frac{1}{2} - \frac{1}{2} \Big( 1 - 4 \varepsilon  (1 - \varepsilon/2) \Big)^{1/2} \, .
\end{gather*}
Therefore,
\begin{displaymath}
\mathrm{max } \, Y_i \ge 1 - \delta_1 (\varepsilon) \Rightarrow U \ge 1- \varepsilon \Rightarrow \mathrm{max } \, Y_i \ge 1- \delta_2 (\varepsilon) \, ,
\end{displaymath}
which implies (\ref{eq_4.2}). The asymptotic relation (\ref{eq_4.3}) is an immediate consequence of (\ref{eq_4.2}).\\

\par For small values $\alpha - \alpha_\mathrm{max}$ the density $f_U$ can increase near its end-point $u =1$. Then already
\begin{equation*}
       1 - F_U (1 - \varepsilon) = O \left( \varepsilon^{\alpha - \alpha_\mathrm{max} }\right) \, , \ \varepsilon \to 0 \, , \textrm{ entails } f_U (1 - \varepsilon) = O \left( \varepsilon^{\alpha - \alpha_\mathrm{max} -1 }\right) \, ,
\end{equation*}
which implies  (\ref{eq_4.4}). \  \\[-3ex]
\end{proof}

\par Now with the Jacobi polynomials $G_k (\alpha n , \alpha n \textrm{; } x)$  from (A.S.22.3.3), belonging to the weight function $w (x) = x^{\alpha n -1}, 0 < x \le 1$, and the shifted Legendre polynomials $P_k^\ast (x) = P_k (2x-1) = {2k \choose k} G_k (1 , 1; x)$,  (A.S.22.2.11), (A.S.22.5.2), we obtain with (\ref{eq_2.2}), (\ref{eq_2.6}) and
\begin{displaymath}
      \int_0^1 \left( 1 - F_C (x) \right) x^{\alpha n +j-1}  dx = \frac{\gamma_{\alpha n +j}}{\alpha n +j}
\end{displaymath}
the following two representations of the cdf  $F_U (u)= 1 - F_C \left( n^{-1/2} u^{-1}\right)$:

\begin{thm}  Let $U$ be defined as in \emph{(\ref{eq_1.3})}. Then
\begin{gather}
       \!\!\!\!\!\!\!\!\!\!\!\!\!\! F_U (u) = \label{eq_4.6}  \\
       \! \!\!\!\!\!\!\!\!\!\!\!\!\!\sum_{k=0}^\infty (-1)^k (\alpha n + k) {\alpha n + 2k \choose k}
        \!\!\left( \sum_{j=0}^k ( -1)^j \,  \frac{(\alpha n + j)_k}{j! (k-j)! }  \frac{\gamma_{\alpha n +j}}{\alpha n +j} \right)
       G_k (\alpha n , \alpha n ; n^{-1/2} u^{-1})\nonumber \, ,
\end{gather}
which is uniformly and absolutely convergent on compact intervals $\lbrack a ; b\rbrack \subset \left(1/\sqrt{n}; 1\right)$ under the first condition in \emph{(\ref{eq_4.5})}. Besides,
\begin{gather}
     \!\!\!\!\!\!\!\!\!\!\!\!\!\! F_U (u) =  \left( n^{1/2} u\right)^{\alpha n-1} \cdot \label{eq_4.7}\\
        \!\!\!\!\!\!\!\!\!\!\!\!\!\!
        \sum_{k=0}^\infty (-1)^k (2 k + 1)
        \left( 
        {\sum_{j=0}^k} (-1)^j
        \displaystyle{\frac{(k - j + 1)_{2j}}{j!j!}}
        \displaystyle{\frac{\gamma_{\alpha n +j}}{\alpha n +j}}  \right)
        \!\!P_k^\ast \! \left( n^{-1/2} u^{-1} \right) \, ,     \nonumber
\end{gather}
which is uniformly and absolutely convergent at least for $n \!\! >$max$\left(  2, 1 \!+ \!\frac{1}{2 \alpha}, \frac{3}{2 \alpha} \right)$.
\end{thm}
\begin{proof}  There remains only an explanation for the additional condition \linebreak $n  \!>  3/(2\alpha)$  in the $2^{\textrm{nd}}$ formula. The $2^{\textrm{nd}}$ series is directly obtained from the corresponding series for the function $x^{\alpha n -1} (1-F_C (x)),\, 0 < x < 1$,  by the substitution $x= n^{-1/2} u^{-1}$.  The square integrability of the density $f_C (x)$ on $\lbrack n^{-1/2}; 1 \rbrack$  is equivalent to the square integrability of $f_U$. The additional condition for $n$ is sufficient for the square integrability of $x^{\alpha n -2}$. Hence, the derivative of the above function is square integrable under the given condition.
\end{proof}
In particular for $\alpha = 1$,  it follows with the moments $\gamma_j$  from (\ref{eq_3.7}) that
\begin{gather}
        \!\!\!\!\!\!\!\!\!\!\!\!\!\!F_U \!\!\!  ~\left( n^{-1/2} x^{-1}\right) =
       1 - F_C (x) = \nonumber \\
       \!\!\!\!\!\!\!\!\!\!\!\!\!\! \sum_{k=0}^\infty (-1)^k (2k +1)
       \left(  
       \sum_{j=0}^k (-1)^j \displaystyle{\frac{(k-j+1)_{2j}}{j!j!}
       \frac{\gamma_{j+1}}{j+1}} \right) \!\!P_k^\ast (x) \, . \label{eq_4.8}
\end{gather}
However, it should be noted that the moments $\gamma_j , j=1, \ldots, n-1$, are computed by infinite series, whereas the coefficients in the orthogonal series (\ref{eq_4.6}), (\ref{eq_4.7}) are given by finite calculations only. This is important since a high number of digits for the moments $\gamma_{\alpha n +j}$ is required to obtain sufficiently accurate values of $F_U$.\\

\par From (\ref{eq_1.2}) we can also obtain the moments
\begin{equation}
       \mu_{2k} = E (U^{2k}) = \frac{E (Z^k)}{E (Y^{2k})} = \frac{k!}{(\alpha n)_{2k}} \sum_{(k)} \prod_{i=1}^n \frac{(\alpha)_{2k_i}}{k_i!} \, , \label{eq_4.9}
\end{equation}
where $\sum_{(k)}$ means summation over all decompositions $k_1 + \ldots + k_n = k $, \linebreak $k_1, \ldots , k_n \in \mathbb{N}_0$ . Since only partial sums are used in the computing procedure (\ref{eq_2.1}), these moments can be computed by the same method, starting with the divergent series
\begin{displaymath}
       \sum_{k=0}^\infty (\alpha)_{2k} \frac{z^k}{k!} \, .
\end{displaymath}
The square integrability of the densities $f_U$ and $f_{U^2}$ is equivalent. Therefore, we get a further convergent orthogonal series
\begin{gather}
 \!\!\!\!\!\!\!\!\!\!\!\!\!\!1 - F_{U^2} (x) = \label{eq_4.10}\\
  \!\!\!\!\!\!\!\!\!\!\!\!\!\!
  \sum_{k=0}^\infty (-1)^k (2k +1)
       \left( 
       \sum_{j=0}^k(-1)^j
      \displaystyle{ \frac{(k-j+1)_{2j}}{j!j!}
       \frac{\mu_{2(j+1)}}{j+1}}\right) \!\!P_k^\ast (x) \, , \
        n^{-1} \le x \le 1 \, ,  \nonumber
\end{gather}
but the extreme magnitude of the sums in (\ref{eq_4.9}) is unfavourable.\\

\par A generalization of the orthogonal series in this section to $U$ computed from $n$ completely independent gamma random variables with not identical order parameters is feasible. However, the required moments would have to be computed from a product of $n$ generating power series instead from just an $n^{\textrm{th}}$ power.

\section{Integral representations of the cumulative distribution function of $U^2$}
Under the $2^{\textrm{nd}}$ condition in (\ref{eq_4.5}) the density of $U_n^2 - n^{-1}$  vanishes at its end-points and is of bounded variation. Then the cdf is representable by an integrated Fourier-sine series , i.e.
\begin{equation}
       P \{ U_n^2 -n^{-1} \le x \} = \frac{2}{\pi} \sum_{k=1}^\infty \Im \! \left( \hat{f}_{U_n^2 -n^{-1}} (t_k)\right)
       \frac{1-\cos (t_k x)}{k}  , \
        t_k = \frac{k \pi}{1-n^{-1}} \, ,   \label{eq_5.1}
\end{equation}
which is absolutely und uniformly convergent with terms at least of order $o(k^{-2})$.\\
\par In a similar way a series for $P \left \{ \log ( \sqrt{n}U ) \le x \right\}$ can be derived from (\ref{eq_3.6}), replacing $s$ by $it_k = ik\pi/\log (\sqrt{n})$. However, an accurate computation of the coefficients is difficult for large $k$. Also a direct use of the Fourier inversion formula is not recommended because of the comparatively slow decrease of the cf of $\log (\sqrt{n}U)$.
\par Before we return to (\ref{eq_5.1}) a univariate integral representation of \mbox{$P \{ U_n^2 \le x \}$} is given by means of parabolic cylinder functions $D_{-\alpha}$ (A.S.19.3.1).

\begin{thm}
Let $X_1, \ldots , X_n$  be completely independent random variables with gamma densities $(\Gamma (\alpha_j))^{-1} e^{-x} x^{\alpha_j-1}$, $\alpha = \alpha_1 + \ldots + \alpha_n$ and $U_n^2 = \left( \textstyle\sum_{i=1}^n X_i \right)^{-2} \textstyle\sum_{i=1}^n X_i^2$. Then the cdf of $U_n^2$  is given by
\begin{gather*}
       \!\!\!\!\!\!\!\!\!\!\!\!\!\!P \{ U_n^2 \le x \} =
     \pi^{-3/2} \sqrt{2} \Gamma (\alpha) x^{(\alpha-1)/2} \cdot \\
     \!\!\!\!\!\!\! \!\!\!\!\!\!\!\int_0^\infty \Re\left(\left( {\textstyle \prod_{j=1}^n \textrm{exp} \!\left( -\frac{is^2}{4} \right) \!D_{-\alpha_j}(i \sqrt{i} s )}  \right)
       \!\textrm{exp} \!\left( {\textstyle \frac{is^2}{4x}} \right) \! D_{-\alpha} (\sqrt{-i}sx^{-1/2}) \sqrt{i} \right) \!ds \, ,
 \end{gather*}
at least for $ \alpha > 1$.
\end{thm}

{\noindent Remarks: The theorem should be true for all $\alpha>0$,  but the additional condition implies the absolute convergence of the integral and facilitates the proof. The cdf of Greenwood's statistic is obtained by identical $\alpha_j =1$ and $\alpha = n$.  For the relation of  $D_{-1}$ to the error function see (\ref{eq_3.2}). Splitting the integral into integrals over intervals between successive zeros of the integrand provides  an alternating series  which enables a good control of the remainder in many cases. E.g. by integration over the intervals (0; 2.978235860548) and (2.978235860548; 3.004569229995) we find the inequalities}
\begin{displaymath}
       0.98999999977 < P \{ 60 U_{60}^2 < 2.7167772982\} < 0.99000000060 \, .
\end{displaymath}
With increasing $n$ a higher computing accuracy is required to compensate the extinction of the leading digits, in particular for the upper tail of this distribution.

\begin{proof} With (A.S.19.3.1), (A.S.19.12.3) and (A.S.13.5.1) we find for all po\-sitive $\alpha$  that
\begin{gather}
       D_{-\alpha} (i \sqrt{i}s) = O (s^{\mathrm{max}  (-\alpha , \alpha - 1)}) \, , \  s \to \infty \, ,\   \label{eq_5.2}\\
       \ \nonumber\\
       \exp ( -is^2/4) D_{-\alpha}(\sqrt{-i}s) \simeq \exp (i \alpha \pi/4) s^{-\alpha}\, , \  s \to \infty \, ,  \label{eq_5.3}
        \\[1.5ex]
       \textrm{(see also (A.S.19.8.1)) }. \nonumber
\end{gather}
Hence, the above integrand is absolutely $\displaystyle{O \! \left( s^{-k-2\sum_{\alpha_ j < 1 \slash 2}{\alpha_j}}\right)}$, where $k =$\linebreak $ \# \{ \alpha_j  |  \alpha_j \! \ge \! 1/2 \}$, and the integral is absolutely convergent if max $\alpha_j \ge $\linebreak $ 1/2$  or  $\alpha > 1/2$.\\
\par Moreover, to justify a change of the order of integration in (\ref{eq_5.6}) according to Fubini,
\begin{gather*}
      \int_0^\infty \int_0^\infty \left( \textstyle\prod_{j=1}^n | D_{-\alpha_j} (i \sqrt{i}s) | \right) | D_{-\alpha} ( x^{-1/2} s \sqrt{-i})  | x^{(\alpha - 1)/2} e^{-tx} dsdx =\\
      \int_0^1 \left( \int_0^\infty \left(  \textstyle\prod_{j=1}^n | D_{-\alpha_j} (i \sqrt{i}s) | \right) | D_{-\alpha} ( x^{-1/2} s \sqrt{-i}) | ds \right) x^{(\alpha - 1)/2} e^{-tx} dx \ +\\
       \int_1^\infty  \left( \int_0^\infty \left( \textstyle\prod_{j=1}^n | D_{-\alpha_j}  (i \sqrt{i}s \sqrt{x}) | \right) | D_{-\alpha} (s \sqrt{-i})| ds \right)
x^{\alpha/2} e^{-tx} dx <\infty
\end{gather*}
is verified by (\ref{eq_5.2}) and (\ref{eq_5.3}).\\
\par The Lt of
\begin{gather}
        \frac{x^{(\alpha-1)/2}}{\sqrt{2\pi}}
       \exp \! \left( - \frac{     x^{-1} s^2 e^{-2i\phi}    }{     4     } \right)
       D_{-\alpha} \! \left( x^{-1/2} se^{-i\phi}  \right)\nonumber\\
       \textrm{ with regard to $x$ is given by}  \nonumber\\
       \exp\!  \left( -se^{-i\phi} \sqrt{2t} \right)
       (2t)^{-(\alpha+1)/2}\, , \ t >0 \, , \ 0 \le \phi < \pi/4 \, . \label{eq_5.4}
\end{gather}
Due to dominated convergence this relation holds also for the limit $\phi = \pi/4$,   i.e.
\begin{gather}
        \frac{x^{(\alpha-1)/2}}{\sqrt{2\pi}}\exp \! \left( \frac{is^2}{4x} \right) D_{-\alpha} \! \left(x^{-1/2} s \sqrt{-i}  \right) \nonumber\\
          \textrm{ has the Lt }  \exp (is \sqrt{2it}) (2t)^{-(\alpha+1)/2}\, .  \label{eq_5.5}
\end{gather}
Therefore,
\begin{gather}
     \frac{x^{(\alpha-1)/2}}{\sqrt{2\pi}} \cdot \nonumber \\
       \int_0^\infty  \left(
      \prod_{j=1}^n
      \exp \! \left( - \frac{is^2}{4}  \right) D_{-\alpha_j}
      \! \left( i \sqrt{i}s  \right)
        \right)
     \exp \! \left( \frac{is^2}{4x}  \right) D_{-\alpha}
     \! \left( x^{-1/2} s \sqrt{-i} \right) \sqrt{i}\ ds \nonumber \\
     \textrm{ has the Lt}
     \label{eq_5.6}\\
     \int_0^\infty
         \left(
      \prod_{j=1}^n
      \exp \! \left( - \frac{is^2}{4}  \right) D_{-\alpha_j}
      \! \left( i \sqrt{i}s  \right)  \right)
     \exp (is\sqrt{2it}) (2t)^{-(\alpha +1)/2} \sqrt{i}\ ds  \, .    \nonumber
\end{gather}
By (A.S.19.12.3) in conjunction with (A.S.13.1.5)  we obtain
\begin{equation}
   (2t)^{-\alpha/2} \exp \left( \frac{(1+is)^2}{8t} \right) D_{-\alpha}
   \left( \frac{1+is}{\sqrt{2t}}\right) = O \left( | 1 + is |^{-\alpha} \right) \, , \ s \to \infty \, ,
\end{equation}
since  $\Re \Bigl( (1+is)^2  \Bigr) < 0 \  if  \ |s| > 1$.\\

\par From the Lt in (\ref{eq_3.1}) we get the Lt  $f_{Y,Z}^\ast (s,t)$ of the joint density $f_{Y,Z}$ of  $Y = \textstyle\sum_{j=1}^n X_j$ and $Z = \textstyle\sum_{j=1}^n X_j^2$ ,  and find by the Laplace inversion formula with regard to $s$ the Lt of $f_{Y,Z}$ with regard to $Z$ as
\begin{gather*}
            \frac{1}{2\pi} \int_{-\infty}^\infty
            \left( \prod_{j=1}^n \exp \!
            \left( \frac{(1+is)^2}{8t} \right)  D_{-\alpha_j} \!
            \left(  \frac{(1+is)}{\sqrt{2t}}  \right) \right)
            (2t)^{-\alpha/2} e^{isy}ds = \\
            \frac{1}{2\pi} \int_{-\infty}^\infty
            \left( \prod_{j=1}^n \exp \!
            \left( \frac{1}{4}  \left(  \frac{1}{\sqrt{2t}} + is \right)^2 \right) \!\! D_{-\alpha_j}\!
            \left( \frac{1}{\sqrt{2t}} + is  \right)
            \right)\!\!
            (2t)^{-(\alpha-1)/2} e^{is\sqrt{2t}y}ds =\\
            e^{-y} \frac{1}{2\pi i} \int_{\frac{1}{\sqrt{2t}}-i\infty}^{\frac{1}{\sqrt{2t}}+i\infty}
             \Big( \prod_{j=1}^n
             e^{\zeta^2/4}
             D_{-\alpha_j} (\zeta) \Big)
             (2t)^{-(\alpha-1)/2} e^{\zeta \sqrt{2t}y} d\zeta =\\
             e^{-y} \frac{1}{2\pi} \int_{-\infty}^\infty
             \Big( \prod_{j=1}^n
              e^{-s^2/4} D_{\alpha_j} (is) \Big)
              (2t)^{-(\alpha-1)/2} e^{is\sqrt{2t}y} ds \, .
\end{gather*}

Then the Lt of the conditional density of  $ZY^{-2} | Y = y$ is
\begin{gather*}
       \frac{\Gamma (\alpha)}{2\pi}  \int_{-\infty}^\infty
       \Big( \prod_{j=1}^n
       e^{-s^2/4} D_{-\alpha_j} (is)
       \Big)
       (2t)^{-(\alpha-1)/2}
       e^{is\sqrt{2t}} ds =\\
       \frac{\Gamma (\alpha)}{\pi}  \int_0^\infty \Re
       \left( \Big( \prod_{j=1}^n
       e^{-s^2/4} D_{-\alpha_j} (is)
       \Big)  (2t)^{-(\alpha-1)/2}
       e^{is\sqrt{2t}}
       \right) ds \, ,
\end{gather*}
independent of $y$. Hence, the Lt of the cdf of $U_n^2$  is
\begin{gather}
      \frac{2}{\pi} \Gamma (\alpha) \int_0^\infty \Re
       \left( \Big( \prod_{j=1}^n
       e^{-s^2/4} D_{-\alpha_j} (is)
        \Big) (2t)^{-(\alpha+1)/2}
       e^{is\sqrt{2t}}
       \right) ds = \nonumber \\
       \ \label{eq_5.8}\\
       \frac{2}{\pi} \Gamma (\alpha) \int_0^\infty \Re
       \left( \Big(\prod_{j=1}^n
       e^{-i s^2/4}  D_{-\alpha_j} (i\sqrt{i}s)  \Big)
       (2t)^{-(\alpha+1)/2}
       e^{is\sqrt{2it}} \sqrt{i}
       \right) ds \, . \nonumber
\end{gather}
Now, theorem 2 follows by comparison of (\ref{eq_5.8}) with (\ref{eq_5.6}).
\end{proof}

   With a sufficient number of terms the partial sums of the integrand in the following theorem show a rather smooth behaviour within the ``main part'', also for larger $n$. The part of the integrand oscillating around zero becomes absolutely small and is more distant from zero. Its contribution to the integral becomes neglectable. For simplicity we give here only the formula derived from identically distributed gamma random variables. The generalization to different order parameters $\alpha_j$  is obvious.

\begin{thm} Let $X_1, \ldots , X_n$  be i.i.d. random variables with a gamma density  $(\Gamma(\alpha))^{-1} e^{-x} x^{\alpha-1}$ and $U_n^2 = \left( \textstyle\sum_{i=1}^n X_i \right)^{-2}  \textstyle\sum_{i=1}^n X_i^2 $. Then the cdf of $U_n^2 - n^{-1}$ is given by
\begin{gather}
      P \left\{  U_n^2 - n^{-1} \le x  \right\} = \nonumber \\
      \frac{\Gamma (\alpha n +1 )}
      {\pi^2 ( \alpha n/e)^{\alpha n} }
      \lim_{K \to \infty} \int_{-\infty}^\infty \sum_{k=1}^K
      \Im \left\{
      \left( \psi (s,t_k)\right)^n
      \exp \left( - i \alpha n s - i (n- 1)^{-1} k \pi \right)
       \right\} \cdot \nonumber \\
       \quad \frac{1- \cos \bigl( (1-n^{-1})^{-1} k \pi x \bigr)}{k} ds \, ,   \label{eq_5.9} \\
       \  \nonumber \\
   \psi (s,t) = \left( \frac{i}{2t} \right)^{\alpha/2}
       \exp \! \left(\frac{i(1 - is)^2}{8t} \right) D_{-\alpha} \!
       \left( \sqrt{i} \frac{1- is}{\sqrt{2t}}\right) \, , \nonumber \\
       t_k = \frac{k\pi}{\alpha^2 n (n-1)}\, , n > \mathrm{max}
 \left( 3,1 + \frac{1}{\alpha}\right) \, , \ 0 \le x \le 1 - \frac{1}{n} \, . \nonumber
\end{gather}
\end{thm}
The limit can be taken under the integral at least for
\begin{equation}
      \left[ \frac{n}{2}  \right] > \mathrm{max}  \left(  2, 1 + \frac{1}{2\alpha} ,  \frac{3}{2\alpha}  \right) \, .
       \label{eq_5.10}
\end{equation}
In particular with $\alpha = 1$  and $t_k = k\pi/\big(n (n-1)\big)$ we have at least for $n\ge 6$:
\begin{gather*}
      \!\!\!\!\!\!\!\!\!\!\!\!P \left\{  U_n^2 - n^{-1} \le x  \right\} =  \frac{n!}{\pi^2 \left( n/e \right)^n} \ \cdot \nonumber \\
     \!\!\!\!\!\!\!\!\!\!\!\!\int_{-\infty}^\infty \sum_{k=1}^\infty
      \Im \left\{
      \left(
      \left( \frac{i}{2t_k} \right)^{1/2} D_{-1}\!\!
      \left( \sqrt{i} \frac{1- is}{\sqrt{2t_k}}\right)
      \exp \!  \left( i \left(
      \frac{(1 - is)^2}{8t_k}-t_k -s \right)
      \right)
     \right)^n
       \right\} \, .
\end{gather*}
\begin{equation}
       \frac{1- \cos \bigl( (1-n^{-1})^{-1} k \pi x \bigr)}{k} ds \, . \label{eq_5.11}
\end{equation}
Remark: The condition (\ref{eq_5.10}) is too restrictive but enables the use of Parseval's identity in the following proof. The absolute integrability of the cf  $\hat{f}_{Y,Z}$ is sufficient for the application of the Fourier inversion formula but not necessary.
\begin{proof} From (\ref{eq_3.1}) we obtain the characteristic functions
\begin{gather*}
       \psi (s,t) = E \bigl(  \exp ( is X + it X^2)  \bigr) =
       \left( \frac{i}{2t}\right)^{\alpha/2}
       \exp \left( \frac{i (1-is)^2}{8t}\right) D_{-\alpha}
       \left( \sqrt{i} \frac{1-is}{\sqrt{2t}} \right) \\
        \textrm{and  } \ \hat{f}_{Y,Z} (s,t) =  (\psi (s,t))^n \, .
\end{gather*}
With  $\Re \bigl( i ( 1 - is)^2 (2t)^{-1} \bigr) = st^{-1}$, $t >0$ and $\alpha n > 1$  the absolute convergence of
\begin{displaymath}
         \int_{-\infty}^0  (\psi (s,t))^ne^{-isy} ds
\end{displaymath}
follows from (A.S.19.12.3) and (A.S.13.5.1) and the absolute convergence of
\begin{displaymath}
         \int_0^{\infty}  (\psi (s,t))^ne^{-isy} ds
\end{displaymath}
is obtained  by the asymptotic relation (A.S.19.8.1). \\

\par Since the conditional distribution of $ZY^{-2}| Y=y$ doesn't depend on $y$ we can choose $y = E(Y) = \alpha n$  and obtain the cf of  $U_n^2 -n^{-1}$:
\begin{displaymath}
         \hat{f}_{U_n^2 -n^{-1}} (t) = ( 2 \pi g_{\alpha n} (\alpha n) )^{-1}
         \exp (-itn^{-1}) \int _{-\infty}^\infty  (\psi (s,(\alpha n)^{-2} t))^n
         \exp  (-i \alpha ns) ds \, .
\end{displaymath}
Together with (\ref{eq_5.1}) this entails (\ref{eq_5.9}).\\

\par Now, a sufficient condition for the square integrability of $\hat{f}_{Y,Z}$ is given. The square integrability of the densities $f_U$ and $f_{U^2}$  is equivalent because of
\begin{displaymath}
       \int_{n^{-1}}^1 \left( f_{U^2} (x)\right)^2  dx=
       \int_{n^{-\frac{1}{2}}}^1 (2u)^{-1} \left( f_U (u)\right)^2 du\, .
\end{displaymath}
According to (\ref{eq_4.5}) $f_U$   is square integrable if  $n > \mathrm{max}  \left( 2, 1 +(2 \alpha)^{-1}\right)$.  From the identity of the distributions of $U^2$ and $ZY^{-2} | Y = y$ it follows
\begin{gather*}
      f_{U^2} (w) = \frac{d}{dw} P \left\{ Z \le wy^2 | Y = y \right\} =
      y^2 \frac{f_{Y,Z} (y,wy^2)}{g_{\alpha n} (y)}
      \\ \textrm{with } g_{\alpha n} (y) = (\Gamma (\alpha n))^{-1} \exp (-y) y^{\alpha n-1} \, .
\end{gather*}
With the additional assumption $n> 3/(2 \alpha)$  we obtain
\begin{gather*}
       \int_0^\infty \left(
       \int_{n^{-1} y^2}^{y^2}
       \left( f_{Y,Z} (y,z)  \right)^2 dz
       \right) dy =\\
        \int_0^\infty \left(
        \int_{n^{-1} y^2}^{y^2}
       \left( \left( g_{\alpha n} (y) \right)^{-1} f_{Y,Z} (y,z)\right)^2 dz
       \right)
        \left( g_{\alpha n} (y) \right)^2 dy =\\
        \int_0^\infty \left(
        y^4 \int_{n^{-1}}^1 \left( \left( g_{\alpha n} (y) \right)^{-1}  f_{Y,Z} (y,wy^2)\right)^2 dw\right)
        \left( y^{-1}    g_{\alpha n} (y) \right)^2 dy =\\
        \int_{n^{-1}}^1 \left(   f_{U^2} (w) \right)^2 dw \cdot \int_0^\infty \left( y^{-1} g_{\alpha n} (y) \right)^{2} dy < \infty \, .
\end{gather*}
Thus, due to Parseval's identity, the cf  $\hat{f}_{Y,Z}$ is square integrable and it follows
\begin{displaymath}
      \int_{-\infty}^{\infty} \int_{-\infty}^\infty | \psi (s,t) |^n dsdt < \infty \textrm{ at least for } \left[ \frac{n}{2} \right] > \mathrm{max}  \left( 2,1 + \frac{1}{2\alpha} , \frac{3}{2 \alpha}  \right) \, .
\end{displaymath}
>From (\ref{eq_4.5}) it follows also that the density $f_U$  vanishes at its end-points for these $n$.\\
Then the change of summation and integration in (\ref{eq_5.9}) is justified by
\begin{gather*}
       \sum_{k \ge K} \frac{|\psi (s,t_k)|^n}{k} =
       O \left( \int_{t_K}^\infty | \psi (s,t) |^n dt  \right)\!, \
       K \to \infty \, , \textrm{ and}\\
       \int_{-\infty}^\infty \left( \int_{t_1}^\infty  | \psi (s,t) |^n dt  \right)  ds < \infty \  .
\end{gather*} \ \\[-8ex]
\end{proof}
\newpage
\section{Some $p$-quantiles of $nU_n^2 - 1 $ }
It should be noted that $n$ is the number of squares. In the application for a goodness of fit test the sample size is $k=n-1$. \\
{\scriptsize
\begin{longtable}{p{0.2cm}p{1.2cm}p{1.2cm}p{1.2cm}p{1.2cm}p{1.2cm}p{1.2cm}p{1.2cm}p{1.2cm}p{1.2cm}}

n:  &p=0.005 &p=0.010 &p=0.025 &p=0.050 &p=0.500 &p=0.950 &p=0.975 &p=0.990 &p=0.995
\\ \hline
\\
10 & 0.17789466  &0.20866627   &0.26014429   &0.31111296   &0.72416820   &1.64511506   &1.93318559   &2.33028205   &2.64025028
\\
11  &0.19730682   &0.22863245   &0.28069377   &0.33196840   &0.74212800   &1.64417262   &1.92622396   &2.31678132  & 2.62372056
\\
12  & 0.21512589   &0.24685519  & 0.29933241   &0.35078097   &0.75769056   &1.64103896   &1.91685253   &2.29998228   &2.60265097
\\
13   &0.23155755  &0.26358997   &0.31635339   &0.36787930  &0.77132393   &1.63644125   &1.90598495   &2.28124294   &2.57887125
\\
14   &0.24678209  & 0.27904048   &0.33198974   &0.38352158   &0.78337931   &1.63086853   &1.89422669    &2.26144430    &2.55358516
\\ [3mm]
15    &0.26094814    &0.29337145    &0.34642980    &0.39791268    &0.79412572    &1.62465699    &1.88198346    &2.24116547    &2.52758847
\\
16   &0.27417918   &0.30671949   &0.35982670   &0.41121857   &0.80377307   &1.61804138   &1.86952924   &2.22079055   &2.50140936
\\
17   &0.28657943  & 0.31919848   &0.37230685   &0.42357520   &0.81248776   &1.61118729   &1.85705008   &2.20057436   &2.47539868
\\
18   &0.29823733   &0.33090418   &0.38397570   &0.43509525   &0.82040353   &1.60421247   &1.84467227   &2.18068462  & 2.44978800
\\
19   &0.30922844   &0.34191790   &0.39492200   &0.44587314   &0.82762921   &1.59720104   &1.83248093   &2.16122966   &2.42472760
\\[3mm]
20   &0.31961776   &0.35230915   &0.40522118   &0.45598866   &0.83425434   &1.59021315   &1.82053247   &2.14227701   &2.40031181
\\
21   &0.32946153   &0.36213779   &0.41493779   &0.46550982   &0.84035330   &1.58329171   &1.80886317   &2.12386603   &2.37659630
\\
22   &0.33880877   &0.37145573   &0.42412748   &0.47449503   &0.84598846   &1.57646707   &1.79749507   &2.10601650   &2.35360984
\\
23   &0.34770234   &0.38030822   &0.43283857  & 0.48299478   &0.85121248   &1.56976041   &1.78644011   &2.08873466   &2.33136257
\\
24   &0.35618001   &0.38873496   &0.44111322   &0.49105303   &0.85607024   &1.56318610   &1.77570311   &2.07201741   &2.30985179
\\ [3mm]
25   &0.36427515   &0.39677092   &0.44898846   &0.49870821   &0.86060015   &1.55675354   &1.76528382   &2.05585528   &2.28906598
\\
26   &0.37201740   &0.40444708   &0.45649693   &0.50599415   &0.86483536   &1.55046837   &1.75517847   &2.04023458   &2.26898774
\\
27   &0.37943320   &0.41179099   &0.46366760   &0.51294074   &0.86880461   &1.54433350   &1.74538083   &2.02513884   &2.24959586
\\
28   &0.38654622   &0.41882726   &0.47052622   &0.51957452   &0.87253294   &1.53834976   &1.73588302   &2.01054993   &2.23086677
\\
29   &0.39337772   &0.42557795   &0.47709583   &0.52591915   &0.87604231   &1.53251648   &1.72667611   &1.99644881   &2.21277557
\\ [3mm]
30   &0.39994689   &0.43206291   &0.48339710   &0.53199583   &0.87935202   &1.52683189   &1.71775048   &1.98281610   &2.19529685
\\
31   &0.40627105   &0.43830006   &0.48944863   &0.53782356   &0.88247913   &1.52129341   &1.70909620   &1.96963242   &2.17840515
\\
32   &0.41236596   &0.44430562   &0.49526722   &0.54341952   &0.88543880   &1.51589788   &1.70070319   &1.95687868   &2.16207537
\\
33   &0.41824592   &0.45009432   &0.50086810   &0.54879921   &0.88824450   &1.51064174   &1.69256144   &1.94453628   &2.14628301
\\
34   &0.42392401   &0.45567960   &0.50626510   &0.55397668   &0.89090828   &1.50552119   &1.68466105   &1.93258721   &2.13100433
\\ [3mm]
35   &0.42941215  &0.46107370   &0.51147084   &0.55896474   &0.89344094   &1.50053226   &1.67699239   &1.92101411   &2.11621649
\\
36   &0.43472131  & 0.46628784  &0.51649683   &0.56377506   &0.89585221   &1.49567088   &1.66954609   &1.90980037   &2.10189756
\\
37   &0.43986154   &0.47133233   &0.52135362   &0.56841830   &0.89815084   &1.49093298   &1.66231313   &1.89893010   &2.08802660
\\
38   &0.44484211   &0.47621662   &0.52605089   &0.57290423   &0.90034476   &1.48631451   &1.65528479   &1.88838815   &2.07458363
\\
39   &0.44967158   &0.48094944   &0.53059755   &0.57724183   &0.90244115   &1.48181147   &1.64845273   &1.87816010   &2.06154966
\\ [3mm]
40   &0.45435784   &0.48553886   &0.53500180   &0.58143937   &0.90444655   &1.47741992   &1.64180898   &1.86823224   &2.04890661
\\
41   &0.45890822   &0.48999231   &0.53927124   &0.58550446   &0.90636692   &1.47313604   &1.63534588   &1.85859156   &2.03663735
\\
42   &0.46332950   &0.49431672   &0.54341286   &0.58944415  &0.90820768   &1.46895611   &1.62905616   &1.84922570   &2.02472560
\\
43   &0.46762802   &0.49851849   &0.54743319   &0.59326496   &0.90997382   &1.46487652   &1.62293287   &1.84012296   &2.01315594
\\
44   &0.47180964  &0.50260359   &0.55133827   &0.59697294   &0.91166990   &1.46089377   &1.61696938   &1.83127222   &2.00191375
\\ [3mm]
45   &0.47587985   &0.50657758   &0.55513370   &0.60057371  &0.91330010   &1.45700451   &1.61115938   &1.82266295   &1.99098517
\\
46   &0.47984376   &0.51044565   &0.55882473   &0.60407248   &0.91486830   &1.45320550   &1.60549687   &1.81428516   &1.98035706
\\
47   &0.48370616   &0.51421263   &0.56241624   &0.60747414   &0.91637805   &1.44949360   &1.59997614   &1.80612938   &1.97001698
\\
48   &0.48747154   &0.51788306   &0.56591280   &0.61078322   &0.91783264   &1.44586584   &1.59459172   &1.79818663   &1.95995310
\\
49   &0.49114409   &0.52146118   &0.56931867   &0.61400398   &0.91923511   &1.44231932   &1.58933845   &1.79044838   &1.95015426
\\ [3mm]
50   &0.49472776   &0.52495098   &0.57263786   &0.61714039   &0.92058828   &1.43885128   &1.58421139   &1.78290656   &1.94060981
\\
51   &0.49822626   &0.52835619   &0.57587411   &0.62019618   &0.92189476   &1.43545906   &1.57920584   &1.77555347   &1.93130970
\\
52   &0.50164308   &0.53168034  & 0.57903095   &0.62317484   &0.92315700   &1.43214014   &1.57431733   &1.76838184   &1.92224435
\\
53   &0.50498151   &0.53492672   &0.58211170   &0.62607966   &0.92437724   &1.42889205   &1.56954159   &1.76138474   &1.91340469
\\
54   &0.50824465  &0.53809847   &0.58511947   &0.62891372   &0.92555761   &1.42571247   &1.56487456   &1.75455558   &1.90478209
\\ \newpage
n:  &p=0.005 &p=0.010 &p=0.025 &p=0.050 &p=0.500 &p=0.950 &p=0.975 &p=0.990 &p=0.995
\\ \hline
\\
55   &0.51143542   &0.54119852   &0.58805721   &0.63167993   &0.92670007   &1.42259915   &1.56031236   &1.74788810   &1.89636834
\\
56   &0.51455659   &0.54422965   &0.59092769   &0.63438104   &0.92780645   &1.41954994   &1.55585131   &1.74137635   &1.88815565
\\
57  &0.51761077  &0.54719450   &0.59373352   &0.63701962   &0.92887849   &1.41656277   &1.55148788   &1.73501467   &1.88013661
\\
58   &0.52060045   &0.55009554   &0.59647719   &0.63959812   &0.92991779   &1.41363566  &1.54721870   &1.72879765   &1.87230415
\\
59   &0.52352797   &0.55293513   &0.59916103   &0.64211884   &0.93092586   &1.41076673   &1.54304057   &1.72272016   &1.86465156
\\ [3mm]
60   &0.52639556   &0.55571550   &0.60178727   &0.64458398   &0.93190412   &1.40795414   &1.53895042   &1.71677730   &1.85717243
\\
61   &0.52920533   &0.55843878   &0.60435801   &0.64699559   &0.93285390   &1.40519615   &1.53494533   &1.71096439   &1.84986065
\\
62  & 0.53195929   &0.56110696  &0.60687524   &0.64935565   &0.93377645   &1.40249108   &1.53102248   &1.70527699   &1.84271041
\\
63   &0.53465935   &0.56372196   &0.60934085   &0.65166600   &0.93467296   &1.39983733   &1.52717920   &1.69971082   &1.83571616
\\
64   &0.53730733   &0.56628561   &0.61175665   &0.65392842   &0.93554453   &1.39723335   &1.52341293   &1.69426184   &1.82887259
\\[3mm]
65   &0.53990496   &0.56879962   &0.61412434   &0.65614458   &0.93639222   &1.39467765   &1.51972122   &1.68892617   &1.82217464
\\
66   &0.54245388   &0.57126564   &0.61644556   &0.65831607   &0.93721701   &1.39216880   &1.51610172   &1.68370008   &1.81561748
\\
67   &0.54495567   &0.57368524   &0.61872186   &0.66044443   &0.93801984   &1.38970545   &1.51255218   &1.67858004   &1.80919648
\\
68   &0.54741183   &0.57605993   &0.62095471   &0.66253108   &0.93880159   &1.38728626   &1.50907044   &1.67356265   &1.80290723
\\
69  &0.54982380   &0.57839112   &0.62314553   &0.66457742   &0.93956310   &1.38490998   &1.50565444   &1.66864466   &1.79674550
\\ [3mm]
70   &0.55219293   &0.58068019   &0.62529566   &0.66658476   &0.94030516   &1.38257538   &1.50230220   &1.66382296   &1.79070723
\\
71  & 0.55452055  & 0.58292844   &0.62740638   &0.66855434   &0.94102852   &1.38028129   &1.49901181   &1.65909457   &1.78478855
\\
72   &0.55680790   &0.58513712   &0.62947893   &0.67048737   &0.94173389   &1.37802659   &1.49578145   &1.65445663   &1.77898576
\\
73   &0.55905619   &0.58730741   &0.63151447   &0.67238500   &0.94242196   &1.37581019   &1.49260938   &1.64990641   &1.77329528
\\
74   &0.56126655   &0.58944047   &0.63351413   &0.67424831   &0.94309335   &1.37363103   &1.48949389   &1.64544127   &1.76771370
\\ [3mm]
75   &0.56344010   &0.59153738   &0.63547899   &0.67607835   &0.94374869   &1.37148812   &1.48643338   &1.64105870   &1.76223775
\\
76   &0.56557788   &0.59359919   &0.63741007   &0.67787613   &0.94438855   &1.36938049   &1.48342629   &1.63675627   &1.75686429
\\
77   &0.56768091   &0.59562692   &0.63930836   &0.67964260   &0.94501349   &1.36730719   &1.48047113   &1.63253166   &1.75159031
\\
78   &0.56975017   &0.59762152   &0.64117481   &0.68137869   &0.94562403   &1.36526732   &1.47756645   &1.62838263   &1.74641289
\\
79   &0.57178658   &0.59958393   &0.64301032   &0.68308527   &0.94622067   &1.36326002   &1.47471087   &1.62430704   &1.74132927
\\ [3mm]
80  & 0.57379104   &0.60151503   &0.64481577   &0.68476320   &0.94680390   &1.36128444   &1.47190306   &1.62030282   &1.73633675
\\
81   &0.57576441   &0.60341568   &0.64659200   &0.68641328   &0.94737416   &1.35933977   &1.46914173   &1.61636798   &1.73143278
\\
82   &0.57770753   &0.60528669   &0.64833980   &0.68803629   &0.94793190   &1.35742523   &1.46642564   &1.61250061   &1.72661487
\\
83   &0.57962119   &0.60712887   &0.65005995  &0.68963298   &0.94847753   &1.35554007   &1.46375360   &1.60869887   &1.72188063
\\
84   &0.58150615   &0.60894297   &0.65175320   &0.69120408   &0.94901145   &1.35368355   &1.46112447   &1.60496099   &1.71722779
\\ [3mm]
85   &0.58336317   &0.61072973   &0.65342025   &0.69275028   &0.94953403   &1.35185497   &1.45853712   &1.60128526   &1.71265412
\\
86   &0.58519295   &0.61248984   &0.65506180   &0.69427223   &0.95004564   &1.35005365   &1.45599050   &1.59767002   &1.70815751
\\
87   &0.58699617   &0.61422400   &0.65667851   &0.69577058   &0.95054664   &1.34827892   &1.45348357   &1.59411371   &1.70373589
\\
88   &0.58877351   &0.61593285   &0.65827102   &0.69724596   &0.95103734   &1.34653015   &1.45101532   &1.59061477   &1.69938729
\\
89   &0.59052559   &0.61761702   &0.65983995   &0.69869894   &0.95151808   &1.34480673   &1.44858480   &1.58717175   &1.69510981
\\ [3mm]
90   &0.59225304   &0.61927713   &0.66138588  &0.70013012   &0.95198916   &1.34310806   &1.44619108   &1.58378320   &1.69090160
\\
91   &0.59395645   &0.62091377   &0.66290939   &0.70154003   &0.95245086   &1.34143355   &1.44383326   &1.58044777   &1.68676089
\\
92   &0.59563640   &0.62252750   &0.66441103   &0.70292921   &0.95290349   &1.33978265   &1.44151047   &1.57716411   &1.68268597
\\
93   &0.59729344   &0.62411887   &0.66589134   &0.70429818   &0.95334729   &1.33815482   &1.43922186   &1.57393094   &1.67867517
\\
94   &0.59892811   &0.62568842   &0.66735083   &0.70564743   &0.95378254   &1.33654953   &1.43696663   &1.57074703   &1.67472691
\\ [3mm]
95   &0.60054092   &0.62723664   &0.66879000   &0.70697744   &0.95420949   &1.33496628   &1.43474398   &1.56761117   &1.67083964
\\
96   &0.60213239   &0.62876405   &0.67020932   &0.70828867   &0.95462836   &1.33340457   &1.43255316   &1.56452220   &1.66701186
\\
97   &0.60370298   &0.63027111   &0.67160927   &0.70958158   &0.95503940   &1.33186393   &1.43039344   &1.56147901   &1.66324213
\\
98   &0.60525319   &0.63175829   &0.67299029   &0.71085658   &0.95544282   &1.33034389   &1.42826408   &1.55848051   &1.65952906
\\
99   &0.60678345   &0.63322604   &0.67435281   &0.71211411   &0.95583885   &1.32884400   &1.42616442   &1.55552565   &1.65587129
\\ [3mm]
100   &0.60829421   &0.63467480   &0.67569726   &0.71335456   &0.95622767   &1.32736383   &1.42409376   &1.55261341   &1.65226752
\end{longtable}
}
\section*{References}
\begin{description}
\item[] Abramowitz, M. and Stegun, I. (1968)  \emph{Handbbook of Mathematical Functions}, Dover Publicatons Inc., New York.
\item[] Burrows, P.M. (1979) Selected percentage points of Greenwood's statistic, \emph{Journal of the Royal Statistical Society, Series A}, \textbf{142}, 256-258.
\item[] Currie, I.D. (1981) Further percentage points of Greenwood's statistic, \emph{Journal of the Royal Statistical Society, Series A}, \textbf{144}, 360-363.
\item[] Currie, I.D. and Stephens, M.A. (1986) Relations between statistics for tes\-ting exponentiality and uniformity, \emph{The Canadian Journal of Statistics}, \textbf{14}, 177-180.
\item[] Darling, D.A. (1953) On a class of problems related to the random division of an interval, \emph{Annals of Mathematical Statistics} \textbf{24}, 239-253.
\item[] Does, R.J.M.M. et al. (1988) Approximating the distribution of Greenwood's statistic, \emph{Statistica Neerlandica} \textbf{42}, 153-161.
\item[] Feller, W. (1971) \emph{An Introduction to Probability Theory and Its Applications}, Vol.II, 2nd ed., John Wiley \& Sons, New York.
\item[] Gardner, A. (1952) Greenwood's ``Problem of intervals'': an exact solution for $n = 3$, \emph{Journal of the Royal Statistical Society, Series B}, \textbf{14}, 135-139.
\item[] Ghosh, K. and Jammalamadaka, S.R. (1998) Small sample approximation for spacing statistics, \emph{Journal of  Statistical Planning and Inference} \textbf{69}, 245-261.
\item[] Ghosh, K. and Jammalamadaka, S.R. (2000) Some recent results on inferences based on spacings, in:  Puri, M.L. (Ed.) Asymptotics in Statistics and Probability: Papers in Honor of George Gregory Roussas, VSP
Utrecht, 185-196.
\item[] Hill, I.D. (1979) Approximating the distribution of Greenwood's statistic with Johnson distributions, \emph{Journal of the Royal Statistical Society, Series A}, \textbf{142}, 378-380.
\item[] Kotz, S. et al. (1994)  Continuous Univariate Distributions, Vol.1, 2nd ed., chap. 17, John Wiley \& Sons, Inc., New York.
\item[] Laha, R.G. (1954) On a characterization of the gamma distribution, \emph{Annals of Mathematical Statistics} \textbf{25},
784-787.
\item[] Lukacs, E. (1955) A characterization of the gamma distribution, \emph{Annals of Mathematical Statistics} \textbf{26}, 319-324.
\item[] Moran, P.A.P. (1947) The random division of an interval, Journal of the Royal Statistical Society, Series B, \textbf{9}, 92-98 (Corrigendum (1981) \emph{Journal of the Royal Statistical Society, Series A}, \textbf{144}, 388).
\item[] Moran, P.A.P: (1951) The random division of an interval II, \emph{Journal of the Royal Statistical Society, Series B},  \textbf{13}, 147-150.
\item[] Provost, S.B. (1988) The exact density of a statistic related to the shape parameter of a gamma random variate, \emph{Metrika} \textbf{35}, 191-196.
\item[] Pyke, R. (1965) Spacings (With Discussion), \emph{Journal of the Royal Statistical Society, Series B}, \textbf{27}, 395-449.
\item[] Royen, T. (2007a) Exact distribution of the sample variance from a gamma parent distribution, arxiv:0704.1415 [math.ST].
\item[] Royen, T. (2007b) On the Laplace transform of some quadratic forms and the exact distribution of the sample variance from a gamma or uniform parent distribution, arxiv.0710.5749 [math.ST].
\item[] Stephens, M.A. (1981) Further percentage points for Greenwood's statistic, \emph{Journal of the Royal Statistical  Society, Series A}, \textbf{144}, 364-366.
\item[] Subhash, C.K. and Gupta, R.P. (1988) A Monte Carlo study of some asymptotic optimal tests of exponentiality against positive aging, \emph{Communications in Statistics - Simulation and Computation} \textbf{17}, 803-811.
\end{description}

 \end{document}